\documentclass[runningheads]{llncs}
\usepackage{graphicx}
\usepackage{microtype}
\usepackage{verbatim}
\usepackage{times}
\usepackage{mathtools}
\usepackage{caption}
\usepackage{subcaption}
\usepackage{amsfonts}
\usepackage{amssymb}
\usepackage{enumitem}
\usepackage[enable]{easy-todo}
\usepackage{color,soul}
\usepackage{etoolbox}
\usepackage{appendix}
\usepackage{algorithm}
\usepackage{algpseudocode}
\AtEndEnvironment{proof}{\phantom{}\qed}

\newcommand{\cK}{\mathcal{K}}
\newcommand{\Rd}{\mathbb{R}^d}
\newcommand{\abs}[1]{\vert #1 \rvert}
\newcommand{\bd}[3]{\mathcal{B}_{#1}(#2,#3)}
\newcommand{\scf}{\phi}

\newcommand{\dualscf}{\phi^{*}}

\newcommand{\scfi}{\psi}

\newcommand{\func}{f}
\newcommand{\obj}{F}
\newcommand{\comp}{h}
\newcommand{\sgn}{\operatorname{sgn}}
\newcommand{\grad}{\triangledown}
\newcommand{\norm}[1]{\lVert #1\rVert }
\newcommand{\dualnorm}[1]{\lVert #1\rVert_* }
\newcommand{\inner}[2]{\langle #1,#2\rangle }
\newcommand{\ex}[1]{\mathbb{E}[#1]}

\newcommand{\sequ}[1]{\{#1_t\}}
\newcommand{\proxi}{\mathcal{P}_\cK}
\newcommand{\Grad}{\mathcal{G}_\cK}

\DeclareMathOperator*{\argmin}{\arg\min}

\newtheorem{assumption}{Assumption}
\begin{document}
\title{Adaptive Zeroth-Order Optimisation of Nonconvex Composite Objectives}
%
%
\author{Weijia Shao\inst{1}, Sahin Albayrak\inst{1}}
\authorrunning{W.Shao et al.}
%
\institute{Technische Universit\"at Berlin, Ernst-Reuter-Platz 7 10587, Berlin Germany}
\maketitle              
\begin{abstract}
In this paper, we propose and analyse algorithms for zeroth-order optimisation of non-convex composite objectives, focusing on reducing the complexity dependence on dimensionality. This is achieved by exploiting the low dimensional structure of the decision set using the stochastic mirror descent method with an entropy alike function, which performs gradient descent in the space equipped with the maximum norm. To improve the gradient estimation, we replace the classic Gaussian smoothing method with a sampling method based on the Rademacher distribution and show that the mini-batch method copes with the non-Euclidean geometry. To avoid tuning hyperparameters, we analyse the adaptive stepsizes for the general stochastic mirror descent and show that the adaptive version of the proposed algorithm converges without requiring prior knowledge about the problem.

\end{abstract}
\section{Introduction}
\label{sec1}
In this work, we study the following stochastic optimisation problem
\begin{equation}
    \label{eq:prob}
    \min_{x\in\cK} \{\obj(x)\coloneqq\func(x)+\comp(x)=\mathbb{E}_\xi[\func(x;\xi)+\comp(x)]\},
\end{equation}
where $\func$ is a black-box, smooth, possibly nonconvex function, $\comp$ is a white box convex function, and $\cK\subseteq \Rd$ is a closed convex set. In many real-world applications, $\comp$ and $\cK$ are sparsity promoting, such as the black-box adversarial attack \cite{chen2018ead}, model agnostic methods for explaining machine learning models \cite{natesan2020model} and sparse cox regression \cite{liu2018zeroth}. Despite the low dimensional structure restricted by $\comp$ and $\cK$, standard stochastic mirror descent methods \cite{lan2020first} and the conditional gradient methods \cite{huang2020accelerated} have oracle complexity depending linearly on $d$ and are not optimal for high dimensional problems.

The gradient descent algorithm is dimensionality independent when the first-order information is available \cite{nesterov2003introductory}. For black-box objective functions, stronger dependence of the oracle complexity on dimensionality is caused by the biased gradient estimation \cite{jamieson2012query}. In \cite{wang2018stochastic}, the authors have proposed a LASSO-based gradient estimator for zeroth-order optimisation of unconstrained convex objective functions. Under the assumption of sparse gradients, the standard stochastic gradient descent with a LASSO-based gradient estimator has a weaker complexity dependence on dimensionality. The sparsity assumption has been further examined for nonconvex problems in \cite{balasubramanian2021zeroth}, which proves a similar oracle complexity of the zeroth-order stochastic gradient method with Gaussian smoothing. 

The critical issue of the algorithms mentioned above is the requirement of sparse gradients, which can not be expected in every application. We wish to improve the dependence on dimensionality by exploiting the low dimensional structure defined by the objective function and constraints. For convex problems, this can be achieved by employing the mirror descent method with distance generating functions that are strongly convex w.r.t. $\norm{\cdot}_1$, such as the exponentiated gradient \cite{kivinen1997exponentiated,warmuth2007winnowing} or the $p$-norm algorithm \cite{duchi2015optimal}. However, a few problems arise if we apply these methods directly to optimising nonconvex functions. First, since these methods are essentially the gradient descent in $(\Rd, \norm{\cdot}_\infty)$, the convergence of the mirror descent algorithm requires variance reduction techniques in that space. Existing variance reduction techniques \cite{cutkosky2019momentum,lan2012optimal,pham2020proxsarah} are developed for the standard Euclidean space, and deriving convergence from the equivalence of the norms in $\Rd$ introduces additional complexity depending on $d$ \cite{ghadimi2016accelerated}. Secondly, the exponentiated gradient \cite{kivinen1997exponentiated} method and its extensions \cite{warmuth2007winnowing} work only for decision sets in the form of a simplex or cross-polytope with a known radius. Therefore, they can hardly be applied to general cases. The $p$-norm algorithm is more flexible and has an efficient implementation for $\ell_1$ regularised problems \cite{shalev2011stochastic}. However, handling $\ell_2$ regularised problems with the $p$-norm algorithm is challenging.

The primary contribution of this paper is the introduction and analysis of algorithms for zeroth-order optimisation of nonconvex composite objective functions. To reduce the complexity dependence on dimensionality without assuming sparse gradients, we employ an entropy alike distance generating function in the stochastic mirror descent method (ZO-ExpMD), which performs gradient descent in $(\Rd,\norm{\cdot}_\infty)$. To improve the gradient estimation in that space, we use the mini-batch approach \cite{ghadimi2016mini} and show that the additional complexity introduced by switching the norms depends on $\ln d$ instead of $d$. Furthermore, we replace the gradient estimation methods applied in \cite{balasubramanian2021zeroth} and \cite{shamir2017optimal} with a smoothing method based on the Rademacher distribution. Our analysis shows that the total number of oracle calls required by ZO-ExpMD for finding an $\epsilon$-stationary point is bounded by $\mathcal{O}(\frac{\ln d}{\epsilon^4})$, which improves the complexity bound $\mathcal{O}(\frac{d}{\epsilon^4})$ attained by proximal stochastic gradient descent (ZO-PSGD) \cite{lan2020first}. To avoid tuning parameters, we extend and analyse the adaptive stepsizes \cite{duchi2011adaptive,li2019convergence} for constrained problems with composite objectives. Then we apply the adaptive stepsizes to ZO-ExpMD and show that the same complexity upper bound can be obtained without knowing the smoothness of $\func$. In addition to the theoretical analysis, we also demonstrate the performance of the developed algorithms in experiments on generating contrastive explanations of deep neural networks \cite{NEURIPS2018_c5ff2543}. 

The rest of the paper is organised as follows. Section~\ref{sec2} reviews related work. In section~\ref{sec3}, we present and analyse our algorithms. Section~\ref{sec4} demonstrates the empirical performance of the proposed algorithms. Finally, we conclude our work with some future research directions in Section~\ref{sec5}.

\section{Related Work}
\label{sec2}
Zeroth-order optimisation of nonconvex objective functions has many applications in machine learning, and signal processing \cite{liu2020primer}. Algorithms for unconstrained nonconvex problems have been studied in \cite{ghadimi2013stochastic,lian2016comprehensive,nesterov2017random} and further enhanced with variance reduction techniques \cite{ji2019improved,liu2018zeroth1}. The high dimensional setting has been discussed in \cite{balasubramanian2021zeroth,wang2018stochastic}, in which algorithms with weaker complexity dependence on dimensionality are proposed. In practice, weaker dependence on dimensionality can also be achieved by applying the sparse perturbation techniques introduced in \cite{ohta2020sparse}.

It is popular to solve constrained problems with zeroth-order Frank-Wolfe algorithms \cite{balasubramanian2021zeroth,chen2020frank,huang2020accelerated}, which require the smoothness of the objective functions. We are motivated by the applications of adversarial attack and explanation methods based on the $\ell_1$ and $\ell_2$ regularisation \cite{chen2018ead,NEURIPS2018_c5ff2543,natesan2020model}, for which the objective functions contain non-smooth components. Our work is based on exploiting the low dimensional structure of the decision set, which has been discussed in \cite{gentile2003robustness,kivinen1997exponentiated,langford2009sparse,shalev2011stochastic,warmuth2007winnowing} for online and stochastic optimization of convex functions and further extended for zeroth-order convex optimization in \cite{duchi2015optimal,shamir2017optimal}. To efficiently implement both $\ell_1$ and $\ell_2$ regularised problems, we used an entropy alike function as the distance-generating function in the stochastic composite mirror descent method. Similar versions of the entropy alike function have previously been applied to unconstrained online convex optimisation \cite{cutkosky2017online,orabona2013dimension}. We combine it with the algorithmic ideas of mini-batch \cite{ghadimi2016mini} and adaptive stepsizes \cite{duchi2011adaptive,li2019convergence} to solve nonconvex optimisation problems.
\section{Algorithms and Analysis}
\label{sec3}
We start the theoretical analysis by introducing some important results of zeroth-order stochastic methods in a finite-dimensional vector space $\mathbb{X}$ equipped with an inner product $\inner{\cdot}{\cdot}$ and some norm $\norm{\cdot}$. Based on them, we then construct and analyze our algorithms in $\Rd$.
\subsection{Adaptive Stochastic Composite Mirror Descent}
\label{sec:adascmd}
Similar to the previous works on stochastic nonconvex optimisation \cite{lan2020first}, the following standard properties of the objective function $\func$ are assumed.
\begin{assumption}
\label{asp:smooth}
For any realisation $\xi$, $\func(\cdot;\xi)$ is $G$-Lipschitz and has $L$-Lipschitz continuous gradients with respect to $\norm{\cdot}$, i.e. \[\dualnorm{\grad \func(x;\xi)-\grad \func(y;\xi)}\leq L\norm{x-y},\] for all $x,y\in\mathbb{X}$, which implies
\[
\abs{\func(y;\xi)-\func(x;\xi)-\inner{\grad \func(x;\xi)}{y-x}}\leq \frac{L}{2}\norm{x-y}^2.
\]
\end{assumption}
\begin{assumption}
\label{asp:sg}
For any $x\in\mathbb{X}$, the stochastic gradient at $x$ is unbiased, i.e. \[\ex{\grad \func(x;\xi)}=\grad \func(x).\]
\end{assumption}
Assumption \ref{asp:smooth} and \ref{asp:sg} imply the $G$-smoothness and $L$-smoothness of $\func$ due to the inequalities
\[
\abs{\func(x)-\func(y)}\leq \ex{\abs{\func(x;\xi)-\func(y;\xi)}}\leq G\norm{x-y},
\]
and
\[
\dualnorm{\grad \func(x)-\grad \func(y)}\leq \ex{\dualnorm{\grad \func(x;\xi)-\grad \func(y;\xi)}}\leq L\norm{x-y}.
\]
Our idea is based on the stochastic composite mirror descent (SCMD), which iteratively updates the decision variable following the rule given by 
\begin{equation}
\label{eq:update_md}
x_{t+1}=\arg\min_{x\in \cK}\inner{g_t}{x}+\comp(x)+\eta_{t}\bd{\scf}{x}{x_t},
\end{equation}
where $g_t$ is an estimation of the gradient $\grad f(x_t)$ and $\scf$ is a distance generating function, i.e. $1$-strongly convex w.r.t. $\norm{\cdot}$. 
Define the generalised projection operator
\begin{equation}
\label{eq:project}
\proxi(x,g,\eta)=\arg\min_{y\in \cK}\inner{g_t}{y}+\comp(y)+\eta\bd{\scf}{y}{x}
\end{equation}
and the generalised gradient map
\begin{equation}
\label{eq:grad}
\Grad(x,g,\eta) = \eta(x-\proxi(x,g,\eta)).
\end{equation}
Following the literature on the stochastic optimisation \cite{balasubramanian2021zeroth,lan2020first}, our goal is to find an $\epsilon$-stationary point $x_R$, i.e. $\ex{\norm{\Grad(x_R,\grad f(x_R),\eta_{R})}^2}\leq \epsilon^2$. Given a sequence of estimated gradients, the convergence of SCMD is upper bounded by the following proposition, the proof of which can be found in the appendix.
\begin{proposition}
\label{lemma:omd}
Let $g_1,\dots,g_T$ be any sequence in $\mathbb{X}$, $x_1,\ldots,x_T$ be the sequence generated by \eqref{eq:update_md} with a distance generating function $\scf$. Then, for any $\func$ satisfying assumption~\ref{asp:smooth} and \ref{asp:sg}, we have
\begin{equation}
\label{lemma:md:eq0}
\begin{split}
&\ex{\frac{1}{T}\sum_{t=1}^T\norm{\Grad(x_t,\grad \func(x_t),\eta_{t})}^2}\\
\leq& \frac{6}{T}\sum_{t=1}^T\ex{\sigma_t^2}+\frac{4}{T}\ex{\sum_{t=1}^T\eta_t(\obj(x_t)-\obj(x_{t+1}))}\\
&+\frac{1}{T}\ex{\sum_{t=1}^T\eta_t(2L-\eta_t)\norm{x_{t+1}-x_t}^2},\\
\end{split}
\end{equation}
where we denote by $\ex{\sigma_t^2}=\ex{\dualnorm{g_t-\grad f(x_t)}^2}$ the variance of the gradient estimation.
\end{proposition}
Setting $\eta_1,\ldots,\eta_T=2L$, the convergence of SCMD depends on the convergence of the variance terms $\sequ{\sigma^2}$, which requires variance reduction techniques. 

In practice, it is difficult to obtain prior knowledge about $L$. To avoid the expensive tuning, we propose an adaptive algorithm with a similar convergence guarantee. The idea is similar to the adaptive stepsizes for unconstrained stochastic optimisation \cite{li2019convergence}, which sets $\eta_t=\sqrt{\sum_{s=1}^{t-1}\dualnorm{g_s}^2+\beta}$ for some $\beta>0$ to control the last term in \eqref{lemma:md:eq0}. For composite objectives, $\norm{x_{t+1}-x_t}^2$ depends not only on $g_t$ but also on $\grad \comp(x_{t+1})$, for which we set $\eta_t\propto\sqrt{\sum_{s=1}^{t-1}\norm{\Grad(x_t,g_t,\eta_t)}^2+1}$. To analyse the proposed method, we assume that the feasible decision set is contained in a closed ball.
\begin{assumption}
\label{asp:compact}
There is some $D>0$ such that $\norm{\proxi(x,g,\eta)}\leq D$ holds for all $\eta>0$, $x\in\cK$ and $g\in \mathbb{X}$.
\end{assumption}
Assumption~\ref{asp:compact} is typical in many composite optimisation problems with regularisation terms in their objective functions. In the following lemma, we propose and analyse the adaptive SCMD.
Due to the compactness of the decision set, we can also assume that the objective function takes values from $[0,B]$.
\begin{assumption}
\label{asp:obj_value}
There is some $B>0$ such that $\obj(x)\in [0,B]$ holds for all $x\in\cK$.
\end{assumption}

\begin{lemma}
\label{lemma:adamd}
Assume \ref{asp:smooth}, \ref{asp:sg}, \ref{asp:compact} and \ref{asp:obj_value}. Define sequence of stepsizes  
\begin{equation}
\label{eq:adastep}
\begin{split}
&\alpha_{t}=(\sum_{s=1}^{t-1}\lambda_s^2\alpha_s^2\norm{x_s-x_{s+1}}^2+1)^{\frac{1}{2}}\\
&\eta_t=\lambda\alpha_t.\\
\end{split}
\end{equation}
for some $0<\lambda\leq\lambda_t\leq \kappa$. Furthermore we assume $D\lambda\geq 1$. Then we have 
\begin{equation}
\label{lemma:adamd:eq0}
\begin{split}
&\ex{\frac{1}{T}\sum_{t=1}^T\norm{\Grad(x_t,\grad \func(x_t),\eta_{t})}^2} \leq\frac{13}{T}\sum_{t=1}^T\ex{\sigma_t^2}+\frac{C}{T}.
\end{split}
\end{equation}
where we define $C=33\kappa^2B^2+\frac{16\sqrt{2}L^2D}{\lambda}(1+2D\lambda)$.
\end{lemma}
\paragraph{Sketch of the proof} The proof starts with the direct application of proposition~\ref{lemma:omd}. The focus is then to control the term $\sum_{t=1}^T\eta_t(2L-\eta_t)\norm{x_{t+1}-x_t}^2$. Since the sequence $\sequ{\eta}$ is increasing, we assume that $\eta_t>L$ starting from some index $t_0$. Then we only need to consider those stepsizes $\eta_1,\ldots,\eta_{t_0-1}$. Adding up $\sum_{t=1}^{t_{0}-2}\norm{x_{t+1}-x_t}^2$ yields a value proportional to $\eta_{t_0-1}$. Thus, the whole term is upper bounded by a constant. The complete proof can be found in the appendix.

The adaptive SCMD does not require any prior information about the problem, including the assumed radius of the feasible decision set. Similar to SCMD with constant stepsizes, its convergence rate depends on the sequence of $\sequ{\sigma^2}$, which will be discussed in the next subsections. 
\subsection{Two Points Gradient Estimation}
\label{sec:2p}
In \cite{balasubramanian2021zeroth}, the authors have proposed the two points estimation with Gaussian smoothing for estimating the gradient, the variance of which depends on $(\ln d)^2$. We argue that the logarithmic dependence on $d$ can be avoided. Our argument starts with reviewing the two points gradient estimation in the general setting. Given a smoothing parameter $\nu>0$, some constant $\delta>0$ and a random vector $u\in\mathbb{X}$, we consider the two points estimation of the gradient given by 
\begin{equation}
\label{eq:grad_est}
\begin{split}
\grad \func_\nu(x)=\mathbb{E}_u[\frac{\delta}{\nu}(\func(x+\nu u)-\func(x))u].
\end{split}
\end{equation}
To derive a general bound on the variance without specifying the distribution of $u$, we make the following assumption.
\begin{assumption}
\label{asp:sample}
Let $\mathcal{D}$ be a distribution with $\operatorname{supp}(\mathcal{D})\subseteq\mathbb{X}$. For $u\sim \mathcal{D}$, there is some $\delta>0$ such that 
\[
\mathbb{E}_u[\inner{g}{u}u]=\frac{g}{\delta}. 
\]
\end{assumption}
Given the existence of $\grad f(x,\xi)$, assumption~\ref{asp:sample} implies $\mathbb{E}_u[\inner{\grad \func(x,\xi)}{u}\delta u]=\grad f(x,\xi)$. Together with the smoothness of $\func(\cdot,\xi)$, we obtain an estimation of $\grad \func(\cdot,\xi)$ with a controlled variance, which is described in the following lemma. Its proof can be found in the appendix.
\begin{lemma}
\label{lemma:bias}
Let $C$ be the constant such that $\norm{x}\leq C\dualnorm{x}$ holds for all $x\in\mathbb{X}$.
Then the follows inequalities hold for all $x\in\mathbb{X}$ and $\func$ satisfying assumptions~\ref{asp:smooth}, \ref{asp:sg} and \ref{asp:sample}.
\begin{enumerate}[label=\alph*)]
\item $\dualnorm{\grad \func_\nu(x)-\grad \func(x)}\leq \frac{\delta\nu C^2L}{2}\mathbb{E}_u[\dualnorm{u}^3]$.
\item $\mathbb{E}_u[\dualnorm{\grad \func_\nu(x;\xi)}^2]\leq  \frac{C^4L^2\delta^2\nu^2}{2}\mathbb{E}_u[\dualnorm{u}^6]+2\delta^2\mathbb{E}_u[\inner{\grad \func(x;\xi)}{u}^2\dualnorm{u}^2]$.
\end{enumerate}
\end{lemma}
For a realisation $\xi$ and a fixed decision variable $x_t$, $\ex{\sigma_t^2}$ can be upper bounded by combining the inequalities in lemma~\ref{lemma:bias}. While most terms of the upper bound can be easily controlled by manipulating the smoothing parameter $\nu$, it is difficult to deal with the term $\delta^2\mathbb{E}_u[\inner{\grad \func(x;\xi)}{u}^2\dualnorm{u}^2]$. Intuitively, if we draw $u_1,\ldots,u_d$ from i.i.d. random variables with zero mean, $\delta^{-2}$ is related the variance. However, small $\ex{\dualnorm{u}^k}$ indicates that $u_i$ must be centred around $0$, i.e. $\delta$ has to be large. Therefore, it is natural to consider drawing $u$ from a distribution over the unit ball with controlled variance. For the case $\dualnorm{\cdot}=\norm{\cdot}_\infty$, this can be achieved by drawing $u_1\ldots,u_d$ from i.i.d. Rademacher random variables.

\subsection{Mini-Batch Composite Mirror Descent for Non-Euclidean Geometry}
With the results in subsections~\ref{sec:adascmd} and \ref{sec:2p}, we can construct an algorithm in $\Rd$, starting with analyzing the gradient estimation based on the Rademacher distribution.
\begin{lemma}
\label{lemma:variance}
Suppose that $\func$ is $L$-smooth w.r.t. $\norm{\cdot}_2$ and $\ex{\norm{\grad f(x)-\grad f(x,\xi)}^2_2}\leq \sigma^2$ for all $x\in\cK$. Let $u_{1},\ldots,u_{d}$ be independently sampled from the Rademacher distribution and
\begin{equation}
\label{lemma:variance:eq0}
g_\nu(x;\xi)=\frac{1}{\nu}(f(x+\nu u;\xi))-f(x;\xi))u
\end{equation}
be an estimation of $\grad f(x)$. Then we have
\begin{equation}
\label{lemma:variance:eq1}
    \begin{split}
        \ex{\norm{g_\nu(x;\xi)-\grad\func_\nu(x)}_\infty^2}\leq  \frac{3\nu^2d^2L^2}{2}+10\norm{\grad \func(x)}_2^2+8\sigma^2.\\
    \end{split}
\end{equation}
\end{lemma}
The dependence on $d^2$ in the first term of \eqref{lemma:variance:eq1} can be removed by choosing $v\propto \frac{1}{d}$, while the rest depends only on the variance of the stochastic gradient and the squared $\ell_2$ norm of the gradient. The upper bound in \eqref{lemma:variance:eq1} is better than the bound $(\ln d)^2(\norm{\grad f(x)}^2_1+\norm{\grad f(x;\xi)-\grad f(x)}^2_1)$ attained by Gaussian smoothing \cite{balasubramanian2021zeroth}. Note that $g_\nu(x;\xi)$ is an unbiased estimator of $\grad f_\nu(x)$. Averaging $g_\nu(x;\xi)$ over a mini-batch can significantly reduce the variance alike quantity in $(\Rd,\norm{\cdot}_\infty)$, which is proved in the next lemma. 
\begin{lemma}
Let $X_1,\ldots,X_m$ be independent random vectors in $\Rd$ such that $\ex{X_i}=\mu$ and $\ex{\norm{X_i-\mu}^2_\infty}\leq \sigma^2$ hold for all $i=1,\ldots,m$. For $ d\geq e$, we have
\label{lemma:vr}
\begin{equation}
    \label{eq:vr}
    \begin{split}
    \ex{\norm{\frac{1}{m}\sum_{i=1}^mX_i-\mu}^2_\infty}\leq \frac{e(2\ln d-1)\sigma^2}{m}.
    \end{split}
\end{equation}
\end{lemma}
\begin{proof}
We first prove the inequality $\ex{\norm{\sum_{i=1}^mX_i-\mu}^2_p}\leq me(\ln d-1)\sigma^2$ for $p=2\ln d$. From the assumption $d\geq e$ and $p=2\ln d$, it follows that the squared $p$ norm is $2p-2$ strongly smooth \cite{orabona2015generalized}, i.e. 
\begin{equation}
    \label{lemma:vr:eq1}
    \norm{x+y}_p^2\leq \norm{x}_p^2+\inner{g_x}{y}+(p-1)\norm{y}^2_p,
\end{equation}
for all $x,y\in\Rd$ and $g_x\in\partial \norm{\cdot}_p^2(x)$. Using the definition of $p$-norm, we obtain
\begin{equation}
    \label{lemma:vr:eq2}
\begin{split}
(p-1)\norm{y}^2_p=&(p-1)(\sum_{i=1}^d\abs{y_i}^p)^{\frac{2}{p}}\\
\leq& (p-1)d^{\frac{2}{p}}\norm{y}_\infty^2\\
\leq& e(2\ln d-1)\norm{y}_\infty^2.\\
\end{split}
\end{equation}
Combining \eqref{lemma:vr:eq1} and \eqref{lemma:vr:eq2}, we have
\begin{equation}
    \label{lemma:vr:eq3}
\begin{split}
    \norm{x+y}_p^2\leq \norm{x}_p^2+\inner{g_x}{y}+e(2\ln d-1)\norm{y}_\infty^2.
\end{split}
\end{equation}
Next, let $X$ and $Y$ be independent random vectors in $\Rd$ with $\ex{X}=\ex{Y}=0$. Using \eqref{lemma:vr:eq3}, we have
\begin{equation}
    \label{lemma:vr:eq4}
\begin{split}
    \ex{\norm{X+Y}_p^2}\leq& \ex{\norm{X}_p^2}+\ex{\inner{g_X}{Y}}+e(2\ln d-1)\ex{\norm{Y}_\infty^2}\\
     = &\ex{\norm{X}_p^2}+\inner{\ex{g_X}}{\ex{Y}}+e(2\ln d-1)\ex{\norm{Y}_\infty^2}\\
     = &\ex{\norm{X}_p^2}+e(2\ln d-1)\ex{\norm{Y}_\infty^2},\\
\end{split}
\end{equation}\
Note that $X_1-\mu,\ldots,X_m-\mu$ are i.i.d. random variable with zero mean. Combining \eqref{lemma:vr:eq4} with a simple induction, we obtain
\begin{equation}
    \label{lemma:vr:eq5}
\begin{split}
    \ex{\norm{\sum_{i=1}^m(X_i-\mu)}_\infty^2}\leq\ex{\norm{\sum_{i=1}^m(X_i-\mu)}_p^2}\leq me(2\ln d-1)\sigma^2.\\
\end{split}
\end{equation}
The desired result is obtained by dividing both sides by $m^2$.
\end{proof}

\begin{algorithm}
	\caption{Zeroth-Order Exponentiated Mirror Descent}
    \label{alg:ExpMD}
	\begin{algorithmic}
	\Require $m>0$, $\nu>0$, $x_1$ arbitrary and a sequence of positive values $\sequ{\eta}$
	\State Define $\scf:\Rd\to \mathbb{R}, x\mapsto \sum_{i=1}^d((\abs{x_i}+\frac{1}{d})\ln(d\abs{x_i}+1)-\abs{x_i})$
	\For{$t=1,\ldots,T$}
	    \State Sample $u_{t,j,i}$ from Rademacher distribution for $j=1,\ldots m$ and $i=1,\ldots d$
	    \State $g_t\coloneqq\frac{1}{m\nu}\sum_{j=1}^{m}(\func(x_t+\nu u_{t,j};\xi_{t,j})-\func(x_t;\xi_{t,j}))u_{t,j}$
	    \State $x_{t+1}=\arg\min_{x\in \cK}\inner{g_t}{x}+\comp(x)+\eta_{t}\bd{\scf}{x}{x_t}$
    \EndFor
    \State Sample $R$ from uniform distribution over $\{1,\ldots,T\}$.
    \State \textbf{Return} $x_{R}$
    \end{algorithmic}
\end{algorithm}
Our main algorithm, which is described in algorithm~\ref{alg:ExpMD}, uses an average of estimated gradient vectors
\begin{equation}
    \label{eq:mb}
    g_t=\frac{1}{m\nu}\sum_{j=1}^{m}(\func(x_t+\nu u_{t,j};\xi_{t,j})-\func(x_t;\xi_{t,j}))u_{t,j},
\end{equation}
and the potential function given by 
\begin{equation}
    \label{eq:reg}
\scf:\Rd\to \mathbb{R}, x\mapsto \sum_{i=1}^d((\abs{x_i}+\frac{1}{d})\ln(d\abs{x_i}+1)-\abs{x_i})
\end{equation}
to update $x_{t+1}$ at iteration $t$. The next lemma proves its strict convexity.
\begin{lemma}
\label{lemma:property}
For all $x, y \in \Rd$, we have 
\[
\scf(y)-\scf(x)\geq  \inner{\grad\scf(x)}{y-x}+\frac{1}{\max\{\norm{x}_1,\norm{y}_1\}+1}\norm{y-x}_1^2
\]
\end{lemma}
The proof of lemma~\ref{lemma:property} can be found in the appendix. If the feasible decision set is contained in an $\ell_1$ ball with radius $D$, then the function $\scf$ defined in \eqref{eq:reg} is $\frac{2}{D+1}$-strongly convex w.r.t $\norm{\cdot}_1$. With $\scf$, update~\eqref{eq:update_md} is equivalent to mirror descent with stepsize $\frac{2\eta_t}{D+1}$ and the distance-generating function $\frac{D+1}{2}\scf$. The performance of algorithm~\ref{alg:ExpMD} is described in the following theorem. 
\begin{theorem}
\label{thm:main}
Assume \ref{asp:smooth}, \ref{asp:sg} for $\norm{\cdot}=\norm{\cdot}_2$, \ref{asp:compact} for $\norm{\cdot}=\norm{\cdot}_1$ and \ref{asp:obj_value}. Furthermore, let $\func$ be $G$-Lipschitz continuous w.r.t. $\norm{\cdot}_2$. Then running algorithm~\ref{alg:ExpMD} with $m=2Te(2\ln d-1)$, $\nu=\frac{1}{d\sqrt{T}}$ and  $\eta_1=,\ldots,=\eta_T=L(D+1)$ guarantees
\begin{equation}
\label{lemma:main:eq0}
\begin{split}
\ex{\norm{\Grad(x_R,\grad \func(x_R),2L)}^2_1}\leq&\sqrt{\frac{2e(2\ln d-1)}{mT}}(6V+4LB),\\
\end{split}
\end{equation}
where we define $V=\sqrt{10G^2+8\sigma^2+2L^2}$.
Furthermore, setting 
\[
\begin{split}
\lambda_t=&\frac{2}{\max\{\norm{x_s}_1,\norm{x_{s+1}}_1\}+1}\\
\alpha_t=&(\sum_{s=1}^{t-1}\lambda_s^2\alpha_s^2\norm{x_{s+1}-x_s}^2_1+1)^\frac{1}{2},
\end{split}
\]
we have 
\[
\begin{split}
&\ex{\frac{1}{T}\sum_{t=1}^T\norm{\Grad(x_t,\grad \func(x_t),\eta_{t})}_1^2} \leq13V\sqrt{\frac{2e(2\ln d-1)}{mT}}+\frac{C}{T}.
\end{split}
\]
where we define $C=132B^2+40\sqrt{2}L^2D(D+1)$.
\end{theorem}
\begin{proof}
First, we bound the variance of the mini-batch gradient estimation. Define \[
g_{t,i}=\frac{1}{\nu}(\func(x_t+\nu u_{t,j};\xi_{t,j})-\func(x_t;\xi_{t,j}))u_{t,j}.
\] 
Since $g_{t,1},\dots,g_{t,m}$ are unbiased estimation of $\grad f_\nu(x_t)$, we have  
\[
\ex{\norm{g_t-\grad f_\nu(x_t)}_\infty^2}\leq\frac{e(2\ln d-1)}{m}(\frac{3\nu^2d^2L^2}{2}+10\norm{\grad \func(x_t)}_2^2+8\sigma^2).
\]
Using lemma~\ref{lemma:variance} and the distribution of $u$, we obtain
\[
\norm{\grad\func_\nu(x_t)-\grad \func(x_t)}_\infty^2\leq\frac{\nu^2 d^2L^2}{4}.
\]
For $m\geq 2e(2\ln d-1)$, we have
\begin{equation}
\label{lemma:main:eq2}
\begin{split}
\ex{\norm{g_t-\grad f(x_t)}_\infty^2}\leq&2\ex{\norm{g_t-\grad f_\nu(x_t)}_\infty^2}+2\ex{\norm{\grad f_\nu(x_t)-\grad f(x_t)}_\infty^2}\\
\leq&\frac{e(2\ln d-1)}{m}(20\norm{\grad \func(x_t)}_2^2+16\sigma^2)+2\nu^2 d^2L^2\\
\leq&\frac{2e(2\ln d-1)}{m}(10G^2+8\sigma^2)+\frac{2L^2}{T}\\
\leq&\sqrt{\frac{2e(2\ln d-1)}{mT}}(10G^2+8\sigma^2+2L^2).\\
\end{split}
\end{equation}
where the last inequality follows from $m=2Te(2\ln d-1)$.

Next, we analyse constant stepsizes. Note that the potential function defined in \eqref{eq:reg} is $\frac{2}{D+1}$ strongly convex w.r.t. to $\norm{\cdot}_1$.
Our algorithm can be considered as an mirror descent with distance generating function given by $\frac{D+1}{2}\scf$, stepsizes $\frac{2\eta_t}{D+1}=2L$. Applying proposition~\ref{lemma:omd} with stepsizes $2L$, we have
\begin{equation}
\label{lemma:main:eq3}
\begin{split}
\ex{\norm{\Grad(x_R,\grad \func(x_R),2L)}^2_1}\leq&\frac{6}{T}\sum_{t=1}^T\ex{\sigma_t^2}+\frac{4L}{T}(F(x_1)-F^*)\\
\leq&\sqrt{\frac{2e(2\ln d-1)}{mT}}(6V+4LB),\\
\end{split}
\end{equation}
where we define $V=10G^2+8\sigma^2+2L^2$.
To analyze the adaptive stepsizes, lemma~\ref{lemma:adamd} can be applied with distance generating function $\frac{D+1}{2}\scf$, stepsizes $\frac{2\alpha_t}{D+1}$ and
\[
\lambda_t=\frac{2}{\max\{\norm{x_t}_1,\norm{x_{t+1}}_1\}+1}.
\]
It holds clearly $0< \lambda=\frac{2}{D+1}\leq \lambda_t\leq 2=\kappa$. W.l.o.g., we assume $D\geq 1$, which implies $2\geq \lambda D\geq 1$. Then we obtain
\begin{equation}
\begin{split}
&\ex{\frac{1}{T}\sum_{t=1}^T\norm{\Grad(x_t,\grad \func(x_t),\eta_{t})}_1^2} \leq13V\sqrt{\frac{2e(2\ln d-1)}{mT}}+\frac{C}{T}.
\end{split}
\end{equation}
where we define $C=132B^2+40\sqrt{2}L^2D(D+1)$.
\end{proof}
The total number of oracle calls for finding an $\epsilon$-stationary point with constant is upper bounded by $\mathcal{O}(\frac{\ln d}{\epsilon^4})$, which has a weaker dependence on dimensionality compared to $\mathcal{O}(\frac{d}{\epsilon^4})$ achieved by ZO-PSGD \cite{lan2020first}. The adaptive stepsize has slightly worse oracle complexity than the well-tuned constant stepsize. Note that a similar result can be obtained by using distance generating function $\frac{1}{2(p-1)}\norm{\cdot}_p^2$ for $p=1+\frac{1}{2\ln d}$. Since the mirror map at $x$ depends on $\norm{x}_p$, it is difficult to handle the popular $\ell_2$ regulariser. Our algorithm has an efficient implementation for Elastic Net regularisation, which is described in the appendix.
\section{Experiments}
\label{sec4}
We examine the performance of our algorithms for generating the contrastive explanation of machine learning models \cite{NEURIPS2018_c5ff2543}, which consists of a set of positive pertinent (PP) features and a set of pertinent negative (PN) features\footnote{The source code is available at \url{https://github.com/VergiliusShao/highdimzo}}. For a given sample $x_0\in \mathcal{X}$ and machine learning model $f:\mathcal{X}\to \mathbb{R}^K $, the contrastive explanation can be found by solving the following optimisation problem \cite{NEURIPS2018_c5ff2543} 
\[
\begin{split}
\min_{x\in\mathcal{\cK}}\quad &l_{x_0}(x)+\gamma_1\norm{x}_1+\frac{\gamma_2}{2}\norm{x}_2^2.\\
\end{split}
\]
Define $k_0=\arg\max_{i}f(x_0)_i$ the prediction of $x_0$.The loss function for finding PP is given by
\[
l_{x_0}(x)=\max\{\max_{i\neq k_0}f(x)_i-f(x)_{k_0},-\kappa\},
\]
and PN is modelled by the following loss function
\[
l_{x_0}(x)=\max\{f(x_0+x)_{k_0}-\max_{i\neq k_0}f(x_0+x)_{i},-\kappa\},
\]
where $\kappa$ is some constant controlling the lower bound of the loss. In the experiment, we first train a LeNet model \cite{lecun1989handwritten} on the MNIST dataset \cite{lecun1989handwritten} and a ResNet$20$ model \cite{7780459} on the CIFAR-$10$ dataset \cite{krizhevsky2009learning}, which attains a test accuracy of $96\%$, $91\%$, respectively. For each class of the images, we randomly pick $20$ correctly classified images from the test dataset and generate PP and PN for them. We set $\gamma_1=\gamma_2=0.1$ for MNIST dataset, and choose $\{x\in\Rd|0\leq x_i\leq x_{0,i}\}$ and $\{x\in\Rd|x_i\geq 0, x_i+x_{0,i}\leq 1\}$ as the decision set for PP and PN, respectively. For CIFAR-$10$ dataset, we set $\gamma_1=\gamma_2=0.5$. ResNet$20$ takes normalized data as input, and images in CIFAR-$10$ do not have an obvious background colour. Therefore, we choose $\{x\in\Rd|\min\{0,x_{0,i}\}\leq x_i\leq\max\{0,x_{0,i}\}\}$ and $\{x\in\Rd|0\leq (x_i+x_{0,i})\nu_i+\mu_i\leq 1\}$, where $\nu_i$ and $\mu_i$ are the mean and variance of the dimension $i$ of the training data, as the decision set for PP and PN, respectively. The search for PP and PN starts from $x_0$ and the center of the decision set, respectively.

Our baseline method is ZO-PSGD with Gaussian smoothing, the update rule of which is given by
\[
x_{t+1}=\arg\min_{x\in\cK}\inner{g_t}{x}+\comp(x)+\eta_t\norm{x-x_t}^2_2.
\]
We fix the mini-batch size $m=200$ for all candidate algorithms to conduct a fair comparison study. Following the analysis of \cite[Corollary 6.10]{lan2020first}, the optimal oracle complexity $\mathcal{O}(\frac{d}{\epsilon^2})$ of ZO-PSGD is obtained by setting $m=dT$ and $\nu=T^{-\frac{1}{2}}d^{-1}=m^{-\frac{1}{2}}d^{-\frac{1}{2}}$. The smoothing parameters for ZO-ExpMD and ZO-AdaExpMD are set to 
\[
\nu=m^{-\frac{1}{2}}(2e(2\ln d-1))^{\frac{1}{2}}d^{-1}
\]
according to theorem~\ref{thm:main}. For ZO-PSGD, ZO-ExpMD, multiple constant stepsizes $\eta_t\in\{10^i|1\leq i\leq 5\}$ are tested. Figure~\ref{fig:BB-mnist} plots the convergence behaviour of the candidate algorithms with the best choice of stepsizes, averaging over 200 images from the MNIST dataset. Our algorithms have clear advantages in the first 50 iterations and achieve the best overall performance for PN. For PP, the loss attained by ZO-PSGD is slightly better than ExpMD with fixed stepsizes, however, it is worse than its adaptive version. Figure~\ref{fig:BB} plots the convergence behaviour of candidate algorithms averaging over $200$ images from the CIFAR-$10$ dataset, which has higher dimensionality than the MNIST dataset. As can be observed, the advantage of our algorithms becomes more significant. Furthermore, choices of stepsizes have a clear impact on the performances of both ZO-ExpMD and ZO-PSGD, which can be observed in \ref{fig:BB-exp-mnist}, \ref{fig:BB-exp} and figure~\ref{fig:BB-PGD-mnist}, \ref{fig:BB-PGD} in the appendix. Notably, ZO-AdaExpMD converges as fast as ZO-ExpMD with well-tuned stepsizes.
\begin{figure}
\centering
\begin{subfigure}{.5\textwidth}
  \centering
  \includegraphics[width=\linewidth]{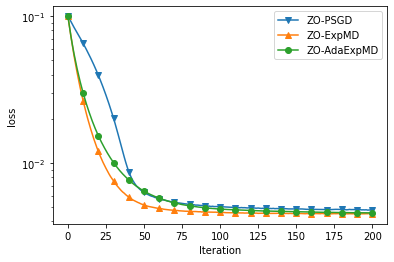}
\caption{Convergence for Generating \textbf{PN}}%
\label{fig:bb-pn-mnist}
\end{subfigure}%
\begin{subfigure}{.5\textwidth}
  \centering
  \includegraphics[width=\linewidth]{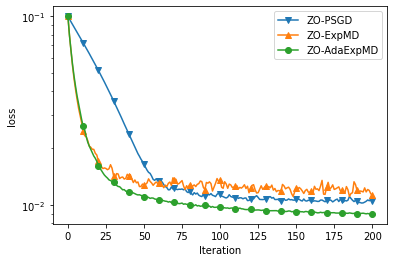}
\caption{Convergence for Generating  \textbf{PP}}%
\label{fig:bb-pp-mnist}
\end{subfigure}
\caption{Black Box Contrastive Explanations on MNIST}
\label{fig:BB-mnist}
\end{figure}
\begin{figure}
\centering
\begin{subfigure}{.5\textwidth}
  \centering
  \includegraphics[width=\linewidth]{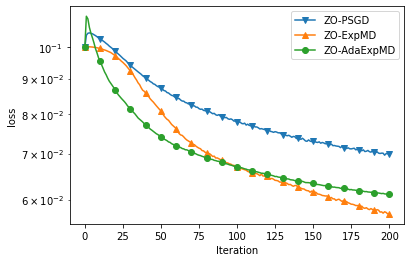}
\caption{Convergence for Generating \textbf{PN}}%
\label{fig:bb-pn}
\end{subfigure}%
\begin{subfigure}{.5\textwidth}
  \centering
  \includegraphics[width=\linewidth]{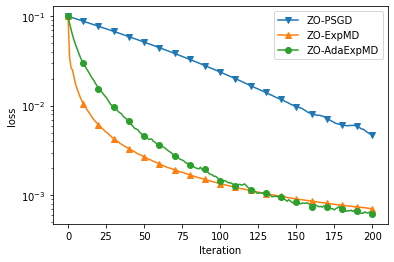}
\caption{Convergence for Generating  \textbf{PP}}%
\label{fig:bb-pp}
\end{subfigure}
\caption{Black Box Contrastive Explanations on CIFAR-$10$}
\label{fig:BB}
\end{figure}

\begin{figure}
\centering
\begin{subfigure}{.5\textwidth}
  \centering
  \includegraphics[width=\linewidth]{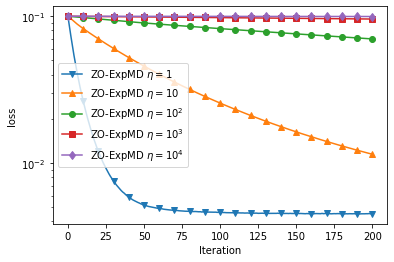}
\caption{Convergence for Generating \textbf{PN}}%
\label{fig:bb-exp-pn-mnist}
\end{subfigure}%
\begin{subfigure}{.5\textwidth}
  \centering
  \includegraphics[width=\linewidth]{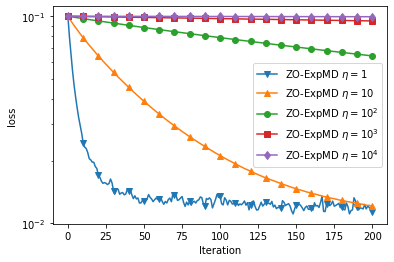}
\caption{Convergence for Generating  \textbf{PP}}%
\label{fig:bb-exp-pp-mnist}
\end{subfigure}
\caption{Impact of step size on ZO-ExpMD on MNIST}
\label{fig:BB-exp-mnist}
\end{figure}
\begin{figure}
\centering
\begin{subfigure}{.5\textwidth}
  \centering
  \includegraphics[width=\linewidth]{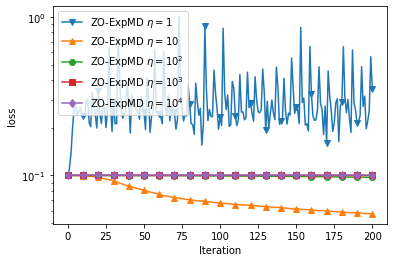}
\caption{Convergence for Generating \textbf{PN}}%
\label{fig:bb-exp-pn}
\end{subfigure}%
\begin{subfigure}{.5\textwidth}
  \centering
  \includegraphics[width=\linewidth]{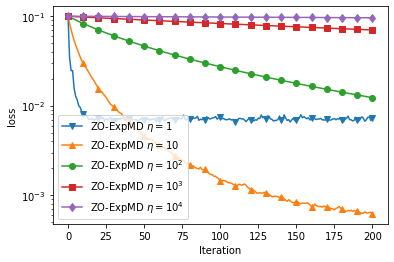}
\caption{Convergence for Generating  \textbf{PP}}%
\label{fig:bb-exp-pp}
\end{subfigure}
\caption{Impact of step size on ZO-ExpMD on CIFAR-$10$}
\label{fig:BB-exp}
\end{figure}
\section{Conclusion}
Motivated by applications in black-box adversarial attack and generating model agnostic explanations of machine learning models, we propose and analyse algorithms for zeroth-order optimisation of nonconvex objective functions. Combining several algorithmic ideas such as the entropy-like distance generating function, the sampling method based on the Rademacher distribution and the mini-batch method for non-Euclidean geometry, our algorithm has an oracle complexity depending logarithmically on dimensionality. With the adaptive stepsizes, the same oracle complexity can be achieved without prior knowledge about the problem. The performance of our algorithms is firmly backed by theoretical analysis and examined in experiments using real-world data. 

Our algorithms can be further enhanced by the acceleration and variance reduction techniques. In the future, we plan to analyse the accelerated version of the proposed algorithms together with variance reduction techniques and draw a systematic comparison with the accelerated or momentum-based zeroth-order optimisation algorithms.
\label{sec5}
\subsubsection*{Acknowledgements}
The research leading to these results received funding from the German Federal Ministry for Economic Affairs and Climate Action under Grant Agreement No. 01MK20002C.
%
%
%
\bibliographystyle{splncs04}
\bibliography{lib}

\begin{thebibliography}{10}
\providecommand{\url}[1]{\texttt{#1}}
\providecommand{\urlprefix}{URL }
\providecommand{\doi}[1]{https://doi.org/#1}

\bibitem{balasubramanian2021zeroth}
Balasubramanian, K., Ghadimi, S.: Zeroth-order nonconvex stochastic
  optimization: Handling constraints, high dimensionality, and saddle points.
  Foundations of Computational Mathematics pp. 1--42 (2021)

\bibitem{chen2020frank}
Chen, J., Zhou, D., Yi, J., Gu, Q.: A frank-wolfe framework for efficient and
  effective adversarial attacks. In: Proceedings of the AAAI conference on
  artificial intelligence. pp. 3486--3494 (2020)

\bibitem{chen2018ead}
Chen, P.Y., Sharma, Y., Zhang, H., Yi, J., Hsieh, C.J.: Ead: elastic-net
  attacks to deep neural networks via adversarial examples. In: Thirty-second
  AAAI conference on artificial intelligence (2018)

\bibitem{cutkosky2017online}
Cutkosky, A., Boahen, K.: Online learning without prior information. In:
  Conference on Learning Theory. pp. 643--677. PMLR (2017)

\bibitem{cutkosky2019momentum}
Cutkosky, A., Orabona, F.: Momentum-based variance reduction in non-convex sgd.
  Advances in neural information processing systems  \textbf{32} (2019)

\bibitem{NEURIPS2018_c5ff2543}
Dhurandhar, A., Chen, P.Y., Luss, R., Tu, C.C., Ting, P., Shanmugam, K., Das,
  P.: Explanations based on the missing: Towards contrastive explanations with
  pertinent negatives. In: Bengio, S., Wallach, H., Larochelle, H., Grauman,
  K., Cesa-Bianchi, N., Garnett, R. (eds.) Advances in Neural Information
  Processing Systems. vol.~31. Curran Associates, Inc. (2018)

\bibitem{duchi2011adaptive}
Duchi, J., Hazan, E., Singer, Y.: Adaptive subgradient methods for online
  learning and stochastic optimization. Journal of Machine Learning Research
  \textbf{12}(Jul),  2121--2159 (2011)

\bibitem{duchi2015optimal}
Duchi, J.C., Jordan, M.I., Wainwright, M.J., Wibisono, A.: Optimal rates for
  zero-order convex optimization: The power of two function evaluations. IEEE
  Transactions on Information Theory  \textbf{61}(5),  2788--2806 (2015)

\bibitem{gentile2003robustness}
Gentile, C.: The robustness of the p-norm algorithms. Machine Learning
  \textbf{53}(3),  265--299 (2003)

\bibitem{ghadimi2013stochastic}
Ghadimi, S., Lan, G.: Stochastic first-and zeroth-order methods for nonconvex
  stochastic programming. SIAM Journal on Optimization  \textbf{23}(4),
  2341--2368 (2013)

\bibitem{ghadimi2016accelerated}
Ghadimi, S., Lan, G.: Accelerated gradient methods for nonconvex nonlinear and
  stochastic programming. Mathematical Programming  \textbf{156}(1-2),  59--99
  (2016)

\bibitem{ghadimi2016mini}
Ghadimi, S., Lan, G., Zhang, H.: Mini-batch stochastic approximation methods
  for nonconvex stochastic composite optimization. Mathematical Programming
  \textbf{155}(1),  267--305 (2016)

\bibitem{7780459}
He, K., Zhang, X., Ren, S., Sun, J.: Deep residual learning for image
  recognition. In: 2016 IEEE Conference on Computer Vision and Pattern
  Recognition (CVPR). pp. 770--778 (2016). \doi{10.1109/CVPR.2016.90}

\bibitem{huang2020accelerated}
Huang, F., Tao, L., Chen, S.: Accelerated stochastic gradient-free and
  projection-free methods. In: International Conference on Machine Learning.
  pp. 4519--4530. PMLR (2020)

\bibitem{iacono2017new}
Iacono, R., Boyd, J.P.: New approximations to the principal real-valued branch
  of the lambert w-function. Advances in Computational Mathematics
  \textbf{43}(6),  1403--1436 (2017)

\bibitem{jamieson2012query}
Jamieson, K.G., Nowak, R., Recht, B.: Query complexity of derivative-free
  optimization. Advances in Neural Information Processing Systems  \textbf{25}
  (2012)

\bibitem{ji2019improved}
Ji, K., Wang, Z., Zhou, Y., Liang, Y.: Improved zeroth-order variance reduced
  algorithms and analysis for nonconvex optimization. In: International
  conference on machine learning. pp. 3100--3109. PMLR (2019)

\bibitem{kivinen1997exponentiated}
Kivinen, J., Warmuth, M.K.: Exponentiated gradient versus gradient descent for
  linear predictors. information and computation  \textbf{132}(1),  1--63
  (1997)

\bibitem{krizhevsky2009learning}
Krizhevsky, A.: Learning multiple layers of features from tiny images. Master's
  thesis, University of Tront  (2009)

\bibitem{lan2012optimal}
Lan, G.: An optimal method for stochastic composite optimization. Mathematical
  Programming  \textbf{133}(1-2),  365--397 (2012)

\bibitem{lan2020first}
Lan, G.: First-order and stochastic optimization methods for machine learning.
  Springer (2020)

\bibitem{langford2009sparse}
Langford, J., Li, L., Zhang, T.: Sparse online learning via truncated gradient.
  Journal of Machine Learning Research  \textbf{10}(3) (2009)

\bibitem{lecun1989handwritten}
LeCun, Y., Boser, B., Denker, J., Henderson, D., Howard, R., Hubbard, W.,
  Jackel, L.: Handwritten digit recognition with a back-propagation network.
  Advances in neural information processing systems  \textbf{2} (1989)

\bibitem{li2019convergence}
Li, X., Orabona, F.: On the convergence of stochastic gradient descent with
  adaptive stepsizes. In: The 22nd International Conference on Artificial
  Intelligence and Statistics. pp. 983--992. PMLR (2019)

\bibitem{lian2016comprehensive}
Lian, X., Zhang, H., Hsieh, C.J., Huang, Y., Liu, J.: A comprehensive linear
  speedup analysis for asynchronous stochastic parallel optimization from
  zeroth-order to first-order. Advances in Neural Information Processing
  Systems  \textbf{29} (2016)

\bibitem{liu2018zeroth}
Liu, S., Chen, J., Chen, P.Y., Hero, A.: Zeroth-order online alternating
  direction method of multipliers: Convergence analysis and applications. In:
  International Conference on Artificial Intelligence and Statistics. pp.
  288--297. PMLR (2018)

\bibitem{liu2020primer}
Liu, S., Chen, P.Y., Kailkhura, B., Zhang, G., Hero~III, A.O., Varshney, P.K.:
  A primer on zeroth-order optimization in signal processing and machine
  learning: Principals, recent advances, and applications. IEEE Signal
  Processing Magazine  \textbf{37}(5),  43--54 (2020)

\bibitem{liu2018zeroth1}
Liu, S., Kailkhura, B., Chen, P.Y., Ting, P., Chang, S., Amini, L.:
  Zeroth-order stochastic variance reduction for nonconvex optimization.
  Advances in Neural Information Processing Systems  \textbf{31} (2018)

\bibitem{natesan2020model}
Natesan~Ramamurthy, K., Vinzamuri, B., Zhang, Y., Dhurandhar, A.: Model
  agnostic multilevel explanations. Advances in neural information processing
  systems  \textbf{33},  5968--5979 (2020)

\bibitem{nesterov2003introductory}
Nesterov, Y.: Introductory lectures on convex optimization: A basic course,
  vol.~87. Springer Science \& Business Media (2003)

\bibitem{nesterov2017random}
Nesterov, Y., Spokoiny, V.: Random gradient-free minimization of convex
  functions. Foundations of Computational Mathematics  \textbf{17}(2),
  527--566 (2017)

\bibitem{ohta2020sparse}
Ohta, M., Berger, N., Sokolov, A., Riezler, S.: Sparse perturbations for
  improved convergence in stochastic zeroth-order optimization. In:
  International Conference on Machine Learning, Optimization, and Data Science.
  pp. 39--64. Springer (2020)

\bibitem{orabona2013dimension}
Orabona, F.: Dimension-free exponentiated gradient. In: NIPS. pp. 1806--1814
  (2013)

\bibitem{orabona2015generalized}
Orabona, F., Crammer, K., Cesa-Bianchi, N.: A generalized online mirror descent
  with applications to classification and regression. Machine Learning
  \textbf{99}(3),  411--435 (2015)

\bibitem{pham2020proxsarah}
Pham, N.H., Nguyen, L.M., Phan, D.T., Tran-Dinh, Q.: Proxsarah: An efficient
  algorithmic framework for stochastic composite nonconvex optimization. J.
  Mach. Learn. Res.  \textbf{21}(110),  1--48 (2020)

\bibitem{shalev2011stochastic}
Shalev-Shwartz, S., Tewari, A.: Stochastic methods for l 1-regularized loss
  minimization. The Journal of Machine Learning Research  \textbf{12},
  1865--1892 (2011)

\bibitem{shamir2017optimal}
Shamir, O.: An optimal algorithm for bandit and zero-order convex optimization
  with two-point feedback. The Journal of Machine Learning Research
  \textbf{18}(1),  1703--1713 (2017)

\bibitem{wang2018stochastic}
Wang, Y., Du, S., Balakrishnan, S., Singh, A.: Stochastic zeroth-order
  optimization in high dimensions. In: International Conference on Artificial
  Intelligence and Statistics. pp. 1356--1365. PMLR (2018)

\bibitem{warmuth2007winnowing}
Warmuth, M.K.: Winnowing subspaces. In: Proceedings of the 24th International
  Conference on Machine Learning. pp. 999--1006 (2007)

\end{thebibliography}
\appendix
\section{Missing Proofs}
\subsection{Proof of Proposition~\ref{lemma:omd}}
\begin{proof}[Proof of Proposition~\ref{lemma:omd}]
First of all, we have
\begin{equation}
\label{lemma:md:eq1}
\begin{split}
&\obj(x_{t+1})-\obj(x_t)\\
\leq &\inner{\grad \func(x_t)+\grad \comp(x_{t+1})}{x_{t+1}-x_t}+\frac{L}{2}\norm{x_{t+1}-x_t}^2\\
\leq &\inner{\eta_t\grad\scf(x_{t+1})-\eta_t\grad\scf(x_{t})}{x_t-x_{t+1}}\\
&+\inner{\grad \func(x_t)-g_t}{x_{t+1}-x_t}+\frac{L}{2}\norm{x_{t+1}-x_t}^2\\
\leq &-\eta_{t}\norm{x_{t+1}-x_t}^2+\inner{\grad \func(x_t)-g_t}{x_{t+1}-x_t}+\frac{L}{2}\norm{x_{t+1}-x_t}^2\\
\leq &-\eta_{t}\norm{x_{t+1}-x_t}^2+\frac{1}{\eta_{t}}\sigma_t^2+\frac{\eta_{t}\norm{x_t-x_{t+1}}^2}{4}+\frac{L}{2}\norm{x_{t+1}-x_t}^2\\
= &-\frac{\eta_{t}}{2}\norm{x_{t+1}-x_t}^2+\frac{1}{\eta_{t}}\sigma_t^2+(\frac{L}{2}-\frac{\eta_{t}}{4})\norm{x_{t+1}-x_t}^2\\
= &-\frac{1}{2\eta_{t}}\norm{\Grad(x_t,g_t,\eta_t)}^2+\frac{1}{\eta_{t}}\sigma_t^2+(\frac{L}{2}-\frac{\eta_{t}}{4})\norm{x_{t+1}-x_t}^2,\\
\end{split}
\end{equation}
where the first inequality uses the $L$-smoothness of $\func$ and the convexity of $\comp$, the second inequality follows from the optimality condition of the update rule, the third inequality is obtained from the strongly convexity of $\scf$ and the fourth line follows from the definition of dual norm. It follows from the $\frac{1}{\eta_t}$ Lipschitz continuity \cite[Lemma 6.4]{lan2020first} of $\proxi(x_t,\cdot,\eta_{t})$ that $\Grad(x_t,\cdot,\eta_{t})$ is $1$-Lipschitz. Thus, we obtain
\begin{equation}
\label{lemma:md:eq2}
\begin{split}
&\norm{\Grad(x_t,\grad \func(x_t),\eta_{t})}^2\\
\leq &2\norm{\Grad(x_t,\grad \func(x_t),\eta_{t})-\Grad(x_t,g_t,\eta_{t})}^2+2\norm{\Grad(x_t,g_t,\eta_{t})}^2\\
\leq &2\sigma_t^2+2\norm{\Grad(x_t,g_t,\eta_{t})}^2\\
\leq &6\sigma_t^2+4\eta_t(\obj(x_t)-\obj(x_{t+1}))+\eta_t(2L-\eta_t)\norm{x_{t+1}-x_t}^2.\\
\end{split}
\end{equation}
Averaging from $1$ to $T$ and taking expectation, we have
\begin{equation}
\label{lemma:md:eq3}
\begin{split}
&\ex{\frac{1}{T}\sum_{t=1}^T\norm{\Grad(x_t,\grad \func(x_t),\eta_{t})}^2}\\ \leq&\frac{6}{T}\sum_{t=1}^T\ex{\sigma_t^2}+\frac{4}{T}\ex{\sum_{t=1}^T\eta_t(\obj(x_t)-\obj(x_{t+1}))}\\
&+\frac{1}{T}\ex{\sum_{t=1}^T\eta_t(2L-\eta_t)\norm{x_{t+1}-x_t}^2},\\
\end{split}
\end{equation}
which is the claimed result.
\end{proof}
\subsection{Proof of Lemma~\ref{lemma:adamd}}
\begin{proof}[Proof of Lemma~\ref{lemma:adamd}]
Applying proposition~\ref{lemma:omd}, we obtain
\begin{equation}
\label{lemma:adamd:eq1}
\begin{split}
&\ex{\frac{1}{T}\sum_{t=1}^T\norm{\Grad(x_t,\grad \func(x_t),\eta_{t})}^2}\\ \leq&\frac{6}{T}\sum_{t=1}^T\ex{\sigma_t^2}+\frac{4}{T}\ex{\sum_{t=1}^T\eta_t(\obj(x_t)-\obj(x_{t+1}))}\\
&+\frac{1}{T}\ex{\sum_{t=1}^T\eta_t(2L-\eta_t)\norm{x_{t+1}-x_t}^2}.\\
\end{split}
\end{equation}
W.l.o.g., we can assume $\obj(x_0)=0$, since it is an artefact in the analysis.
The second term of the upper bound above can be rewritten into
\begin{equation}
\label{lemma:adamd:eq2}
\begin{split}
&\sum_{t=1}^T\eta_t(\obj(x_t)-\obj(x_{t+1}))\\
= &\eta_1\obj(x_0)-\eta_T\obj(x_{T+1})+\sum_{t=1}^T(\eta_t-\eta_{t-1})\obj(x_t)\\
\leq &B\sum_{t=1}^T(\eta_t-\eta_{t-1})\\
\leq &B\eta_T\\
\leq &4\kappa^2B^2+\frac{1}{16\kappa^2}\eta_T^2\\
\end{split}
\end{equation}
where the first inequality follows from $\obj(x_0)=0$ and $\obj(x_{T+1})\geq 0$ and the last line uses the H\"older's inequality.
Using the definition of $\eta_T$, we have
\begin{equation}
    \label{lemma:adamd:eq2.1}
    \begin{split}
    \frac{1}{16\kappa^2}\eta_T^2\leq&\frac{\lambda^2}{16\kappa^2}\sum_{t=1}^T\lambda_s^2\alpha_s^2\norm{x_{s+1}-x_s}^2+\frac{1}{16}    \\
    \leq&\frac{1}{16}\sum_{t=1}^T\norm{\Grad(x_t,g_t,\eta_t)}+\frac{1}{16}\\
    \leq&\frac{1}{8}\sum_{t=1}^T\norm{\Grad(x_t,\grad f(x_t),\eta_t)}+\frac{1}{8}\sum_{t=1}^T\sigma_t^2+\frac{1}{16}\\
    \end{split}
\end{equation}
Next, define
\[
\begin{split}
t_0=\begin{cases}
    \min\{1\leq t\leq T|\eta_t>L\} ,& \text{if }\{1\leq t\leq T|\eta_t>2L\}\neq \emptyset\\
    T,              & \text{otherwise}.
\end{cases}
\end{split}
\]
Then, the third term in \eqref{lemma:adamd:eq1} can be bounded by
\begin{equation}
\label{lemma:adamd:eq3}
\begin{split}
&\sum_{t=1}^T\eta_t(2L-\eta_t)\norm{x_{t+1}-x_t}^2\\
=&\sum_{t=1}^{t_0-1}\eta_t(2L-\eta_t)\norm{x_{t+1}-x_t}^2+\sum_{t=t_0}^T\eta_t(2L-\eta_t)\norm{x_{t+1}-x_t}^2\\
=& 2L\sum_{t=1}^{t_0-1}\frac{\alpha_t\eta_t\norm{x_{t+1}-x_t}^2}{\alpha_t}\\
=& \frac{2\sqrt{2}L}{\lambda}\sum_{t=1}^{t_0-1}\frac{\eta_t^2\norm{x_{t+1}-x_t}^2}{\sqrt{2\sum_{s=1}^{t-1}\lambda_s^2\alpha_s^2\norm{x_{s+1}-x_s}^2+2}}\\
\leq& 4\sqrt{2}LD\sum_{t=1}^{t_0-1}\frac{\eta_t^2\norm{x_{t+1}-x_t}^2}{\sqrt{\sum_{s=1}^{t-1}\lambda_s^2\alpha_s^2\norm{x_{s+1}-x_s}^2+4\lambda^2\alpha_t^2D^2}}\\
\leq& 4\sqrt{2}LD\sum_{t=1}^{t_0-1}\frac{\eta_t^2\norm{x_{t+1}-x_t}^2}{\sqrt{\sum_{s=1}^{t}\lambda^2\alpha_s^2\norm{x_{s+1}-x_s}^2}}\\
\leq& 8\sqrt{2}LD\sqrt{\sum_{t=1}^{t_0-1}\lambda^2\alpha_t^2\norm{x_{t+1}-x_t}^2}\\
\leq& 8\sqrt{2}LD(\alpha_{t_0-1}+2\lambda D\alpha_{t_0-1})\\
\leq& \frac{8\sqrt{2}LD}{\lambda}(1+2D\lambda)\eta_{t_0-1}\\
\leq& \frac{16\sqrt{2}L^2D}{\lambda}(1+2D\lambda),\\
\end{split}
\end{equation}
where we used the assumption $\lambda D\geq 1$ for the first inequality, lemma \ref{lemma:log} for the third inequality and the rest inequalities follow from the assumptions on $\lambda_t$,$D$ and $\eta_{t_0-1}$.  
Combining \eqref{lemma:adamd:eq1}, \eqref{lemma:adamd:eq2}, \eqref{lemma:adamd:eq2.1} and \eqref{lemma:adamd:eq3}, we have
\begin{equation}
\label{lemma:adamd:eq4}
\begin{split}
&\ex{\frac{1}{T}\sum_{t=1}^T\norm{\Grad(x_t,\grad \func(x_t),\eta_{t})}^2}\\ \leq&\frac{13}{T}\sum_{t=1}^T\ex{\sigma_t^2}+\frac{1}{T}(\frac{1}{2}+32\kappa^2B^2)+\frac{16\sqrt{2}L^2D}{\lambda T}(1+2D\lambda).\\
\end{split}
\end{equation}
For simplicity and w.l.o.g., we can assume $\kappa^2B^2\geq \frac{1}{2}$. Define $C=33\kappa^2B^2+\frac{16\sqrt{2}L^2D}{\lambda}(1+2D\lambda)$, we obtain the claimed result.
\end{proof}
\subsection{Proof of Lemma~\ref{lemma:bias}}
\begin{proof}[Proof of Lemma~\ref{lemma:bias}]
Let $\grad \func_\nu(x)$ be as defined in \eqref{eq:grad_est}, then we have
\begin{equation}
\label{lemma:bias:eq3}
\begin{split}
&\dualnorm{\grad \func_\nu(x)-\grad \func(x)}\\
= &\dualnorm{ \mathbb{E}_u[\frac{\delta}{\nu}(\func(x+ \nu u)-\func(x))u]-\grad \func(x)}\\
= &\frac{\delta}{\nu}\dualnorm{ \mathbb{E}_u[(\func(x+\nu u)-\func(x)-\inner{\grad \func(x)}{\nu u})u]}\\
\leq &\frac{\delta}{\nu} \mathbb{E}_u[(\func(x+\nu u)-\func(x)-\inner{\grad \func(x)}{\nu u})\dualnorm{u}]\\
\leq &\frac{\delta\nu C^2L}{2}\mathbb{E}_u[\dualnorm{u}^3]\\
\end{split}
\end{equation}
where the second equality follows from the assumption~\ref{asp:sample}, the third line uses the Jensen's inequality, and the last line follows the $L$ smoothness of $\func$.
Next, we have
\begin{equation}
\label{lemma:bias:eq4}
\begin{split}
&\mathbb{E}_u[{\dualnorm{\grad \func_\nu(x;\xi)}^2}]\\
=&\mathbb{E}_u[\frac{\delta^2}{\nu^2}\abs{\func(x+\nu u;\xi)-\func(x;\xi)}^2\dualnorm{u}^2]\\
=&\frac{\delta^2}{\nu^2}\mathbb{E}_u[(\func(x+\nu u;\xi)-\func(x;\xi)-\inner{\grad \func(x;\xi)}{\nu u}+\inner{\grad \func(x;\xi)}{\nu u})^2\dualnorm{u}^2]\\
\leq&\frac{2\delta^2}{\nu^2}\mathbb{E}_u[(\func(x+\nu u;\xi)-\func(x;\xi)-\inner{\grad \func(x;\xi)}{\nu u})^2\dualnorm{u}^2]\\
&+\frac{2\delta^2}{\nu^2}\mathbb{E}_u[\inner{\grad \func(x;\xi)}{\nu u}^2\dualnorm{u}^2]\\
\leq &\frac{C^4L^2\delta^2\nu^2}{2}\mathbb{E}_u[\dualnorm{u}^6]+2\delta^2\mathbb{E}_u[\inner{\grad \func(x;\xi)}{u}^2\dualnorm{u}^2],\\
\end{split}
\end{equation}
which is the claimed result.
\end{proof}
\subsection{Proof of Lemma~\ref{lemma:variance}}
\begin{proof}[Proof of Lemma~\ref{lemma:variance}]
We clearly have $\ex{uu^\top}=I$. From lemma~\ref{lemma:bias} with the constant $C=\sqrt{d}$ and $\delta=1$, it follows 
\begin{equation}
    \label{lemma:variance:eq2}
    \begin{split}
        \ex{\norm{g_\nu(x;\xi)}_\infty^2}\leq &\ex{\frac{d^2L^2\nu^2}{2}\mathbb{E}_u[\norm{u}_\infty^6]+2\mathbb{E}_u[\inner{\grad \func(x;\xi)}{u}^2\norm{u}_\infty^2]}\\
        \leq &\ex{\frac{d^2L^2\nu^2}{2}+2\mathbb{E}_u[\inner{\grad \func(x;\xi)}{u}^2]}\\
        \leq &\frac{d^2L^2\nu^2}{2}+2\ex{\norm{\grad \func(x;\xi)}_2^2}\\
        \leq &\frac{d^2L^2\nu^2}{2}+4\ex{\norm{\grad \func(x)-\grad \func(x;\xi)}_2^2}+4\norm{\grad \func(x)}_2^2\\
        \leq &\frac{d^2L^2\nu^2}{2}+4\sigma^2+4\norm{\grad \func(x)}_2^2\\
    \end{split}
\end{equation}
where the second inequality uses the fact the $\norm{u}_\infty\leq 1$ and the third inequality follows from the Khintchine inequality. The variance is controlled by
\begin{equation}
    \label{lemma:variance:eq3}
    \begin{split}
        &\ex{\norm{g_\nu(x;\xi)-\grad\func_\nu(x)}_\infty^2}\\
        \leq &2\ex{\norm{g_\nu(x;\xi)}_\infty^2}+2\norm{\grad\func_\nu(x)}_\infty^2\\
        \leq & \nu^2d^2L^2+8(\norm{\grad \func(x)}_2^2+\sigma^2)+2\norm{\grad \func(x)}_\infty^2+ 2\norm{\grad \func(x)-\grad\func_\nu(x)}_\infty^2\\
        \leq & \nu^2d^2L^2+8(\norm{\grad \func(x)}_2^2+\sigma^2)+2\norm{\grad \func(x)}_\infty^2+ \frac{\nu^2d^2 L^2}{2}\\
        \leq & \frac{3\nu^2d^2L^2}{2}+10\norm{\grad \func(x)}_2^2+8\sigma^2,\\
    \end{split}
\end{equation}
which is the claimed result.
\end{proof}
\subsection{Proof of Lemma~\ref{lemma:property}}
\begin{proof}[Proof of Lemma~\ref{lemma:property}]
We first show that each component of $\scf$ is twice continues differentiable. Define $\scfi:\mathbb{R}\mapsto\mathbb{R}:x\mapsto (\abs{x}+\frac{1}{d})\ln(d\abs{x}+1)-\abs{x}$. It is straightforward that $\scfi$ is differentiable at $x\neq 0$ with \[\scfi'(x)= \ln (d\abs{x}+1)\sgn(x).\] For any $h\in\mathbb{R}$, we have
\[
\begin{split}
\scfi(0+h)-\scfi(0)=&(\abs{h}+\frac{1}{d})\ln(d\abs{h}+1)-\abs{h}\\
\leq &(\abs{h}+\frac{1}{d})d\abs{h}-\abs{h}\\
=& dh^2,
\end{split}
\]
where the first inequality uses the fact $\ln x\leq x-1$.
Furthermore, we have 
\[
\begin{split}
\scfi(0+h)-\scfi(0)=&(\abs{h}+\frac{1}{d})\ln(d\abs{h}+1)-\abs{h}\\
\geq & (\abs{h}+\frac{1}{d})(\frac{\abs{h}}{\abs{h}+\frac{1}{d}})-\abs{h}\\
\geq & 0,
\end{split}
\]
where the first inequality uses the farc $\ln x\geq 1-\frac{1}{x}$.
Thus, we have 
\[
0\leq \frac{\scfi(0+h)-\scfi(0)}{h}\leq dh
\]
for $h>0$ and 
\[
 dh\leq \frac{\scfi(0+h)-\scfi(0)}{h}\leq 0
\]
for $h<0$, from which it follows $\lim_{h\to 0} \frac{\scfi(0+h)-\scfi(0)}{h}=0$.
Similarly, we have for $x\neq 0$ \[
\scfi''(x)=\frac{1}{\abs{x}+\frac{1}{d}}.
\]
Let $h\neq 0$, then we have
\[
\frac{\scfi'(0+h)-\scfi'(0)}{h}=\frac{ \ln(d\abs{h}+1)\sgn(h)}{h}=\frac{ \ln(d\abs{h}+1)}{\abs{h}}.
\]
From the inequalities of the logarithm, it follows
\[
\frac{1}{\abs{h}+\frac{1}{d}}\leq\frac{\scfi'(0+h)-\scfi'(0)}{h}\leq d.
\]
Thus, we obtain $\scfi''(0)=d$.
Since $\scfi$ is twice continuously differentiable with $\scfi''(x)>0$ for all $x\in\mathbb{R}$, $\scf$ is strictly convex, and we have, for all $x,y\in \Rd$, there is a $c\in [0,1]$ such that 
\begin{equation}
\label{lemma:property:eq1}    
\begin{split}
\scf(y)-\scf(x)=& \grad \scf(x)(y-x)+\sum_{i=1}^d\frac{1}{\abs{cx_i+(1-c)y_i}+\frac{1}{d}}(x_i-y_i)^2.\\
\end{split}
\end{equation}
For all $v\in \Rd$, we have
\begin{equation}
\label{lemma:property:eq2}    
\begin{split}
    &\sum_{i=1}^d \frac{v_i^2}{\abs{cx_i+(1-c)y_i}+\frac{1}{d}}\\
    =& \sum_{i=1}^d \frac{v_i^2}{\abs{cx_i+(1-c)y_i}+\frac{1}{d}}\frac{\sum_{i=1}^d(\abs{cx_i+(1-c)y_i}+\frac{1}{d})}{\sum_{i=1}^d(\abs{cx_i+(1-c)y_i}+\frac{1}{d})}\\
    \geq &\frac{1}{\sum_{i=1}^d(\abs{cx_i+(1-c)y_i}+\frac{1}{d})}(\sum_{i=1}^d \abs{v_i})^2\\
    \geq &\frac{1}{c\norm{x}_1+(1-c)\norm{y}_1+1}(\sum_{i=1}^d \abs{v_i})^2\\
    = &\frac{1}{\max\{\norm{x}_1,\norm{y}_1\}+1}\norm{v}_1^2,
\end{split}
\end{equation}
where the first inequality follows from the Cauchy-Schwarz inequality. Combining \eqref{lemma:property:eq1} and \eqref{lemma:property:eq2}, we obtain the claimed result.
\end{proof}
\section{Efficient Implementation for Elastic Net Regularization}
We consider the following updating rule
\begin{equation}
\label{eq:omd:impl}
\begin{split}
    y_{t+1} &=\grad \dualscf(\grad \scf(x_t)-\frac{g_t}{\eta_t})\\
    x_{t+1} &=\argmin_{x\in\cK}\comp(x)+\eta_t\bd{\scf}{x}{y_{t+1}}.
\end{split}
\end{equation}
It is easy to verify 
\[
(\grad \dualscf (\theta))_i =(\frac{1}{d}\exp(\abs{\theta_i})-\frac{1}{d})\sgn(\theta_i).
\]
Furthermore, \eqref{eq:omd:impl} is equivalent to the mirror descent update \eqref{eq:update_md} due to the relation
\[
\begin{split}
x_{t+1}=&\argmin_{x\in\cK}\comp(x)+\eta_t\bd{\scf}{x}{y_{t+1}}\\
=&\argmin_{x\in\cK}\comp(x)+\eta_t\scf(x)-\inner{\eta_t\grad\scf(y_{t+1})}{x}\\
=&\argmin_{x\in\cK}\comp(x)+\eta_t\scf(x)-\inner{\eta_t\grad \scf(x_t)-g_t}{x}\\
=&\argmin_{x\in\cK}\inner{g_t}{x}+\comp(x)+\eta_t\bd{\scf}{x}{x_t}.\\
\end{split}
\]
Next, We consider the setting of $\cK=\Rd$ and $\comp(x)=\gamma_1 \norm{x}_1+\frac{\gamma_2}{2}\norm{x}^2_2$. 
The minimiser of 
\[
\comp(x)+\eta_t\bd{\scf}{x}{y_{t+1}}
\]
in $\Rd$ can be simply obtained by setting the subgradient to $0$. For $\ln(d\abs{y_{i,t+1}}+1)\leq\frac{\gamma_1}{\eta_{t+1}}$, we set $x_{i,t+1}=0$. Otherwise, the $0$ subgradient implies $\sgn(x_{i,t+1})=\sgn(y_{i,t+1})$ and $\abs{x_{i,t+1}}$ given by the root of
\[
\begin{split}
\ln(d\abs{y_{i,t+1}}+1)=\ln(d\abs{x_{i,t+1}}+1)+\frac{\gamma_1}{\eta_{t}}+\frac{\gamma_2}{\eta_{t}}\abs{x_{i,t+1}}
\end{split}
\]
for $i=1,\ldots, d$. 
For simplicity, we set $a=\frac{1}{d}$, $b=\frac{\gamma_2}{\eta_{t}}$ and $c=\frac{\gamma_1}{\eta_{t}}-\ln(d\abs{y_{i,t+1}}+1)$. It can be verified that $\abs{x_{i,t+1}}$ is given by
\begin{equation}
\abs{x_{i,t+1}}=\frac{1}{b}W_0(ab\exp(ab-c))-a,
\end{equation}
where $W_0$ is the principle branch of the \textit{Lambert function} and can be well approximated \cite{iacono2017new}.
For $\gamma_2=0$, i.e. the $\ell_1$ regularised problem, $\abs{x_{i,t+1}}$ has the closed form solution
\begin{equation}
\abs{x_{i,t+1}}=\frac{1}{d}\exp(\ln(d\abs{y_{i,t+1}}+1)-\frac{\gamma_1}{\eta_{t}})-\frac{1}{d}.
\end{equation}
The implementation is described in Algorithm \ref{alg:reg}.
\begin{algorithm}
	\caption{Solving $\min _{x\in \Rd} \inner{g_t}{x}+\comp(x)+\eta_t\bd{\scf}{x}{x_{t}}$}
    \label{alg:reg}
	\begin{algorithmic}
	\For{$i=1,\ldots,d$}
	    \State $z_{i,t+1}=\ln (d \abs{x_{i,t}}+1)\sgn(x_{i,t})-\frac{g_{i,t}}{\eta_t}$
	    \State $y_{i,t+1}=(\frac{1}{d}\exp(\abs{z_{i,t+1}})-\frac{1}{d})\sgn(z_{i,t+1})$
	    \If{$\ln(d\abs{y_{i,t+1}}+1)\leq\frac{\gamma_1}{\eta_{t}}$}
	    \State $x_{t+1,i}\gets 0$
        \Else
        \State $a\gets\beta$
        \State $b\gets\frac{\gamma_2}{\eta_{t}}$
        \State $c\gets\frac{\gamma_1}{\eta_{t}}-\ln(d\abs{{y}_{t+1,i}}+1)$
        \State $x_{t+1,i}\gets\frac{1}{b}W_0(ab\exp(ab-c))-a$
        \EndIf
    \EndFor
    \State Return $x_{t+1}$
    \end{algorithmic}
\end{algorithm}
\subsection{Impact of the Choice of Stepsizes of PGD}
\begin{figure}
\centering
\begin{subfigure}{.5\textwidth}
  \centering
  \includegraphics[width=\linewidth]{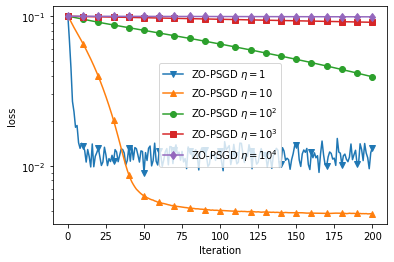}
\caption{Convergence for Generating \textbf{PN}}%
\label{fig:bb-pgd-pn-mnist}
\end{subfigure}%
\begin{subfigure}{.5\textwidth}
  \centering
  \includegraphics[width=\linewidth]{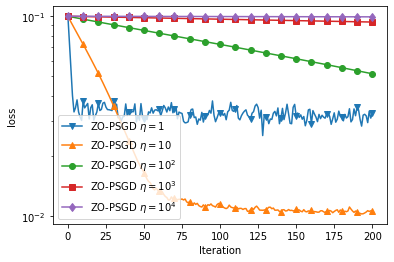}
\caption{Convergence for Generating  \textbf{PP}}%
\label{fig:bb-pgd-pp-mnist}
\end{subfigure}
\caption{Impact of step size on ZO-PSGD on MNIST}
\label{fig:BB-PGD-mnist}
\end{figure}
\begin{figure}
\centering
\begin{subfigure}{.5\textwidth}
  \centering
  \includegraphics[width=\linewidth]{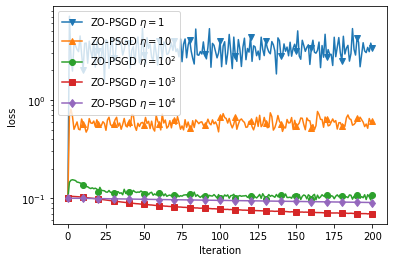}
\caption{Convergence for Generating \textbf{PN}}%
\label{fig:bb-pgd-pn}
\end{subfigure}%
\begin{subfigure}{.5\textwidth}
  \centering
  \includegraphics[width=\linewidth]{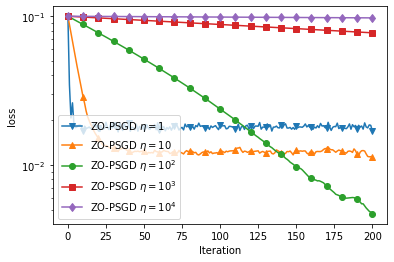}
\caption{Convergence for Generating  \textbf{PP}}%
\label{fig:bb-pgd-pp}
\end{subfigure}
\caption{Impact of step size on ZO-PSGD on CIFAR-$10$}
\label{fig:BB-PGD}
\end{figure}
\begin{lemma}
	\label{lemma:log}
	For positive values $a_1,\ldots,a_n$ the following holds:
	\begin{enumerate}
	\item \[\sum_{i=1}^{n}\frac{a_i}{\sum_{k=1}^{i}a_k+1}\leq \log (\sum_{i=1}^{n}a_i+1)\]
	\item \[\sqrt{\sum_{i=1}^{n}a_i}\leq\sum_{i=1}^{n}\frac{a_i}{\sqrt{\sum_{j=1}^ia_j^2}}\leq 2\sqrt{\sum_{i=1}^{n}a_i}.\]
	\end{enumerate}
	\begin{proof}
    The proof of (1) can be found in Lemma A.2 in \cite{levy2018online}
	For (2), we define $A_0=1$ and $A_i=\sum_{k=1}^{i}a_i+1$ for $i>0$. 
	Then we have 
		\[
		\begin{split}
		\sum_{i=1}^{n}\frac{a_i}{\sum_{k=1}^{i}a_k+1}=&\sum_{i=1}^{n}\frac{A_{i}-A_{i-1}}{A_i}\\
		=&\sum_{i=1}^{n}(1-\frac{A_{i-1}}{A_i})\\
		\leq &\sum_{i=1}^{n}\ln\frac{A_{i}}{A_{i-1}}\\
		= & \ln A_n-\ln A_0\\
		= &\ln \sum_{i=1}^{n}(a_i+1),
		\end{split}
		\]
	where the inequality follows from the concavity of $\log$.
	\end{proof}
\end{lemma}

\end{document}